\documentclass[a4paper, 11pt, reqno]{amsart}

\usepackage{amsmath,amssymb,amsthm}
\usepackage{graphicx}
\usepackage{mathrsfs}
\usepackage{enumerate}
\usepackage{tikz} 
\usepackage{verbatim}
\usepackage{amssymb}
\usepackage{amsbsy}
\usepackage{amscd}
\usepackage{amsmath}
\usepackage{amsthm}
\usepackage{mathrsfs}
\usepackage{eucal}
\usepackage{bbm}

\newtheorem{theorem}{Theorem}[section]
\newtheorem{prop}[theorem]{Proposition}
\newtheorem{cor}[theorem]{Corollary}

\newtheorem{lemma}[theorem]{Lemma}

\theoremstyle{remark}
\newtheorem{remark}[theorem]{Remark}
\newtheorem{example}[theorem]{Example}

\DeclareMathOperator{\Li}{\mathrm{Li}}
\DeclareMathOperator{\LL}{\mathcal{L}}

\DeclareFontEncoding{OT2}{}{}
\DeclareFontSubstitution{OT2}{cmr}{m}{n}
\DeclareSymbolFont{cyss}{OT2}{wncyss}{m}{n}
\DeclareMathSymbol{\sh}{\mathbin}{cyss}{`x}

\numberwithin{equation}{section}

\allowdisplaybreaks

\begin{document}

\title[Two formulas for certain double and multiple polylogarithms]{Two formulas for certain double and multiple polylogarithms in two variables}

\author[M.~Kaneko]{Masanobu Kaneko}
\address{Faculty of Mathematics\\ Kyushu University\\ Motooka 744, Nishi-ku \\ Fukuoka 819-0395\\ Japan}
\email{kaneko.masanobu.661@m.kyushu-u.ac.jp}
\author[H.~Tsumura]{Hirofumi Tsumura}
\address{Department of Mathematical Sciences\\ Tokyo Metropolitan University\\ 1-1, Minami-Ohsawa\\ Hachioji, Tokyo 192-0397\\ Japan}
\email{tsumura@tmu.ac.jp }

\date{}

\maketitle

\begin{abstract}
We give a weighted sum formula for the double polylogarithm in two variables, from which we can recover the classical
weighted sum formulas for double zeta values, double $T$-values, and some double $L$-values. 
Also presented is a connection-type formula for a two-variable multiple polylogarithm, which specializes to previously known
single-variable formulas. This identity can also be regarded as a generalization
of the renowned five-term relation for the dilogarithm.
\end{abstract}

\baselineskip 18pt
\section{Introduction} \label{sec-1}

For positive integers $k_1,\ldots,k_r\  (k_r\ge2)$, the multiple zeta value (MZV) is defined by
$$ \zeta(k_1,k_2,\ldots,k_r)=  \sum_{0<m_1<\cdots <m_r}
\frac1{m_1^{k_1}m_2^{k_2}\cdots m_r^{k_r}}.$$
Among many generalizations of the MZV, the multiple $L$-value (of `shuffle-type') is defined for Dirichlet characters
$\chi_1,\ldots,\chi_r$ as
\begin{align}
&L_\sh(k_1,\ldots,k_r;\chi_1,\ldots,\chi_r) =\sum_{m_1,\ldots,m_r\geq 1}\frac{\chi_1(m_1)\cdots
\chi_r(m_r)}{m_1^{k_1}(m_1+m_2)^{k_2}\cdots (m_1+\cdots+m_r)^{k_r}}.  \label{Def-L-val}
\end{align}
Here, $k_r$ may equal to 1 if $\chi_r$ is a non-trivial character. See \cite{AK2004} for basic properties of $L_\sh$ and its 
companion $L$-value $L_*$ (`stuffle-type').  When all $\chi_i$ are trivial characters $\mathbbm{1}_2$ modulo 2, the $L$-value
$L_\sh(k_1,\ldots,k_r;\mathbbm{1}_2,\ldots,\mathbbm{1}_2)$ is (up to a normalizing factor $2^r$) nothing but the multiple $T$-value (MTV)
\begin{equation}\label{MTV}  
T(k_1,k_2,\ldots,k_r)=2^r\!\!\!\sum_{0<m_1<\cdots <m_r\atop m_i\equiv i\bmod 2}
\frac1{m_1^{k_1}m_2^{k_2}\cdots m_r^{k_r}},
\end{equation} 
which we studied in detail in \cite{KT2020-ASPM, KT2020-Tsukuba}. 

In this paper, we first highlight the so-called weighted sum formula for double zeta, $T$-, and $L$-values. 
The original weighted sum formula for double zeta values given in Ohno-Zudilin~\cite{OZ2008} is
\begin{align}
& \sum_{j=2}^{k-1}2^{j-1}\zeta(k-j,j)=\frac{k+1}{2}\zeta(k) \quad (k\geq 3). \label{WSF-OZ}
\end{align}
An analogous formula for double $T$-values is proved in \cite{KT2020-Tsukuba}:
\begin{align}
& \sum_{j=2}^{k-1}2^{j-1}T(k-j,j)=(k-1)T(k)\quad (k\geq 3). \label{WSF-KT}
\end{align}
Earlier, Nishi proved similar weighted sum formulas for double $L$-values with non-trivial Dirichlet characters
of conductors 3 and 4 (see Proposition~\ref{Prop-4} in Section~\ref{sec-3}).

Our first result of this paper is a `generic' weighted sum formula for a double polylogarithm in two variables,
from which all of the above formulas follow.

The multiple polylogarithm is defined by 
\begin{align}
& \Li_{k_1,\ldots,k_r}(z_1,\ldots, z_r)  =  \sum_{0<m_1<\cdots <m_r}
\frac{z_1^{m_1}\cdots z_r^{m_r}}{m_1^{k_1}\cdots m_r^{k_r}} \label{Def-PL}
\end{align}
for $k_1,\ldots, k_r\in \mathbb{Z}_{\ge1}$ and $z_1,\ldots, z_r\in \mathbb{C}$ with $|z_j|\le 1$ $(1\leq j\leq r)$ ($z_r\ne1$ if
$k_r=1$).

\begin{theorem}\label{Th-1-2}\ For an integer $k\in \mathbb{Z}_{\geq 2}$ and for complex numbers 
$x,y\in \mathbb{C}$ with $|x|\leq 1$, $|y|\leq 1$, $x\neq 1$, $y\neq 1$, (we moreover assume $xy\neq 1$
if $k=2$), we have
\begin{align}
& \sum_{j=1}^{k-1}2^{j-1}\left(\Li_{k-j,j}(x^{-1}y,x)+\Li_{k-j,j}(xy^{-1},y)\right) +\Li_{1,k-1}(x^{-1},xy)+\Li_{1,k-1}(y^{-1},xy)\notag\\
& \quad =(\Li_1(x)+\Li_1(y))\Li_{k-1}(xy)+(k-1)\Li_k(xy). \label{WSF-2}
\end{align}
\end{theorem}
As a corollary, we obtain a one-variable version as follows.

\begin{cor}\label{Cor1-2}\ For $k\in \mathbb{Z}_{\geq 2}$ and $x\in \mathbb{C}$ with $|x|=1$ and $x\neq 1$, 
\begin{align}
& \sum_{j=1}^{k-1}2^{j-1}\left(\Li_{k-j,j}(x^{-2},x)+\Li_{k-j,j}(x^2,x^{-1})\right)-\Li_{k-1,1}(1,x)-\Li_{k-1,1}(1,x^{-1}) \notag\\
& \quad =\Li_k(x)+\Li_k(x^{-1})+(k-1)\zeta(k). \label{WSF-1}
\end{align}
\end{cor}

We next consider the following multiple polylogarithm in two variables:
\begin{align}
\LL_{k_1,\ldots,k_r}(x,y)
&=\sum_{n_1,\ldots,n_r\geq 1}\frac{\prod_{j=1}^{r}~x^{n_j}(1-y^{n_j})}{\prod_{j=1}^{r}\left(\sum_{\nu=1}^{j}n_\nu\right)^{k_j}}\label{def-Chapo}\\
&=\sum_{0<m_1<\cdots<m_r} \frac{x^{m_r}(1-y^{m_2-m_1})(1-y^{m_3-m_2})\cdots (1-y^{m_r-m_{r-1}})}{m_1^{k_1}m_2^{k_2}\cdots m_r^{k_r}},\notag
\end{align}
where $k_1,\ldots,k_r\in \mathbb{Z}_{\geq 1}$ and $x,y\in \mathbb{C}$ with $|x|, |y|\leq 1$ ($x\ne1$ if $k_r=1$).

When $y=0$, this is the usual multiple polylogarithm~\eqref{Def-PL}, and when $y=-1$, this coincides with the level-$2$ multiple polylogarithm
\[ A(k_1,\ldots,k_r;x)=2^r\sum_{0<l_1<\cdots <l_r \atop l_j\equiv j \bmod 2}\frac{x^{l_r}}{l_1^{k_1}\cdots l_r^{k_r}}\]
studied in \cite[Section 4.1]{KT2020-Tsukuba} (see also \cite{Sasaki2012}). 

\begin{remark}
The series \eqref{def-Chapo} was essentially defined by Kamano \cite{Kamano2023} as a polylogarithm corresponding to 
Chapoton's `multiple $T$-value
with one parameter $c$', denoted $Z_c(k_1,\ldots,k_r)$. In  \cite{Chapo2022} Chapoton gave a multiple integral expression
of $Z_c(k_1,\ldots,k_r)$ and deduced its duality relation which generalizes the duality for multiple $T$-values.  
Kamano's multiple polylogarithm with one (fixed) parameter $c$ defined in \cite{Kamano2023} 
is, in our notation, equal to $\LL_{k_1,\ldots,k_r}(x,c)$, and Chapoton's $Z_c(k_1,\ldots,k_r)$ is $\LL_{k_1,\ldots,k_r}(1,c)$.
Kamano further defined and studied poly-Bernoulli numbers associated with $\LL_{k_1,\ldots,k_r}(x,c)$ 
and related zeta functions of so-called Arakawa-Kaneko type. 
\end{remark}

We prove the following. 

\begin{theorem}\label{Th-1-4}\  For integers $r\in \mathbb{Z}_{\geq 1}$, $k\in \mathbb{Z}_{\geq 2}$ and for complex numbers 
$x,y\in \mathbb{C}$ with $|x|\leq 1$, $|y|\leq 1$, $|(1-x)/(1-xy)|\leq 1$, $y\neq 1$, $xy\neq 1$, we have 

\begin{align} 
\LL_{\scriptsize{\underbrace{1,\ldots,1}_{r-1},k}}\left(\frac{1-x}{1-xy},y\right) &= (-1)^{k-1} \sum_{j_1+\cdots+j_k=r+k\atop
\forall j_i\ge1} \LL_{\scriptsize{\underbrace{1,\ldots,1}_{j_k-1}}}\left(\frac{1-x}{1-xy},y\right)\LL_{j_1,\ldots,j_{k-1}}\left(x,y\right)\notag\\
&\quad +\sum_{j=0}^{k-2} (-1)^j \LL_{\scriptsize{\underbrace{1,\ldots,1}_{r-1},k-j}}(1,y)\LL_{\scriptsize{\underbrace{1,\ldots,1}_{j}}}\left(x,y\right).\label{eq-1-4}
\end{align}
\end{theorem}

This provides a generic formula which specializes (when $y=0$) to an Euler-type connection formula for the usual multiple polylogarithm
$\Li_{\scriptsize{1,\ldots,1,k}}(1,\ldots,1,x)$ mentioned in~\cite[Remark 3.7]{Kaneko-Tsumura2018} and (when $y=-1$) to that given by Pallewatta \cite{Pallewatta}. Moreover, this identity in the case $(r,k)=(1,2)$ 
is equivalent to the well-known two-variable, five-term relation for the classical dilogarithm (see Remark~\ref{Rem-5-2}).

We prove Theorem~\ref{Th-1-2} and Corollary~\ref{Cor1-2} in Section~\ref{sec-2} and deduce several known weighted sum formulas in Section~\ref{sec-3}. In Section~\ref{sec-4}, we prove Theorem~\ref{Th-1-4}. 

\section{Proof of Theorem \ref{Th-1-2}} \label{sec-2}

We compute the sum $\sum_{j=1}^{k-1}\Li_{k-j}(x)\Li_{j}(y)$ in two different (`double-shuffle') ways.  

\begin{lemma}\label{Lemma-2-4}\ Let $k\in \mathbb{Z}_{\geq 2}$ and $x,y\in \mathbb{C}$ with 
$|x|\leq 1,\ |y|\leq 1$, $x\neq 1$ and $y\neq 1$. Then 
\begin{align}
\sum_{j=1}^{k-1}\Li_{k-j}(x)\Li_{j}(y)& =\sum_{\mu=1}^{k-1}~2^{\mu-1}\left(\Li_{k-\mu,\mu}(x^{-1}y,x)
+\Li_{k-\mu,\mu}(xy^{-1},y)\right). \label{2-4-1}
\end{align}
\end{lemma}

\begin{proof} \ First we assume $|x|<1$ and $|y|<1$. 
Recall the partial fraction decomposition
\begin{equation}
\frac{1}{m^in^j}=\sum_{\mu=1}^{i+j-1}\left\{ \binom{\mu-1}{i-1}\frac{1}{n^{i+j-\mu}(m+n)^\mu}
+\binom{\mu-1}{j-1}\frac{1}{m^{i+j-\mu}(m+n)^\mu}\right\}\ \, (i,j\ge1)\label{GKZ-2}
\end{equation}
(see for instance \cite[Equation (19)]{GKZ2006}). From this, we have
\begin{align}
\Li_{k-j}(x)\Li_{j}(y)&=\sum_{m,n\geq 1}\frac{x^my^n}{m^{k-j}n^j}  \notag\\
&=\sum_{m,n\geq 1}x^my^n \sum_{\mu=1}^{k-1}\bigg\{ \binom{\mu-1}{k-j-1}\frac{1}{n^{k-\mu}(m+n)^\mu}
+\binom{\mu-1}{j-1}\frac{1}{m^{k-\mu}(m+n)^\mu}\bigg\} \notag\\
&=\sum_{\mu=1}^{k-1}\bigg\{ \binom{\mu-1}{k-j-1}\Li_{k-\mu,\mu}(x^{-1}y,x)
+\binom{\mu-1}{j-1}\Li_{k-\mu,\mu}(xy^{-1},y)\bigg\}. \label{Li-sh}
\end{align}
Using this, we see that the left-hand side of \eqref{2-4-1} is
\begin{align*}
& \sum_{j=1}^{k-1}\sum_{\mu=1}^{k-1}\left\{\binom{\mu-1}{k-j-1}\Li_{k-\mu,\mu}(x^{-1}y,x)+\binom{\mu-1}{j-1}\Li_{k-\mu,\mu}(xy^{-1},y)\right\}\\
& \ \ =\sum_{\mu=1}^{k-1}\left(\Li_{k-\mu,\mu}(x^{-1}y,x)\sum_{j=1}^{k-1}\binom{\mu-1}{k-j-1}+\Li_{k-\mu,\mu}(xy^{-1},y)\sum_{j=1}^{k-1}\binom{\mu-1}{j-1}\right)\\
& \ \ =\sum_{\mu=1}^{k-1}~2^{\mu-1}\left(\Li_{k-\mu,\mu}(x^{-1}y,x)+\Li_{k-\mu,\mu}(xy^{-1},y)\right),
\end{align*}
where we have used the binomial identity $\sum_{j=1}^{k-1}\binom{\mu-1}{k-j-1}=\sum_{j=1}^{k-1}\binom{\mu-1}{j-1}=2^{\mu-1}$
valid when $k-2\geq \mu-1$. It follows from Abel's limit theorem (see \cite[Chap.\,2,\,Theorem\,3 and Remark]{Alf}) that 
\eqref{2-4-1} holds for $x,y\in \mathbb{C}$ with $|x|\leq 1,\ |y|\leq 1$, $x\neq 1$ and $y\neq 1$.
\end{proof}

On the other hand, we have (the stuffle product)
\begin{align}
\Li_{k-j}(x)\Li_j(y)&=\sum_{m,n\geq 1}\frac{x^my^n}{m^{k-j}n^j} =\left(\sum_{0<m<n}
+\sum_{0<n<m}+\sum_{0<m=n}\right)\frac{x^my^n}{m^{k-j}n^j}\\
&=\Li_{k-j,j}(x,y)+\Li_{j,k-j}(y,x)+\Li_{k}(xy). \label{2-2-1}
\end{align}
Now the identity \eqref{Li-sh} in the case of $(i,j)=(k-1,1)$ with $x$ being replaced by $xy$ gives
\begin{align}
\Li_{k-1}(xy)\Li_{1}(y)&=\Li_{1,k-1}(x^{-1},xy)+\sum_{j=1}^{k-1}\Li_{k-j,j}(x,y) \label{2-1-2}
\end{align}
and, exchanging $x$ and $y$ and replacing $j$ by $k-j$, 
\begin{align}
\Li_{k-1}(xy)\Li_{1}(x)&=\sum_{j=1}^{k-1}\Li_{j,k-j}(y,x)+\Li_{1,k-1}(y^{-1},xy). \label{2-1-2-2}
\end{align}
Hence by \eqref{2-2-1}, \eqref{2-1-2}, and \eqref{2-1-2-2}, we have
\begin{align*}
&\sum_{j=1}^{k-1}\Li_{k-j}(x)\Li_{j}(y) =\sum_{j=1}^{k-1}\left(\Li_{k-j,j}(x,y)+\Li_{j,k-j}(y,x)\right)+(k-1)\Li_{k}(xy)\\
&=(\Li_1(x)+\Li_1(y))\Li_{k-1}(xy)-\Li_{1,k-1}(x^{-1},xy)-\Li_{1,k-1}(y^{-1},xy)+(k-1)\Li_{k}(xy).
\end{align*}
We therefore have proved Theorem~\ref{Th-1-2}.

To prove Corollary~\ref{Cor1-2}, first rewrite the term $(\Li_1(x)+\Li_1(y))\Li_{k-1}(xy)$ on the right-hand side of
\eqref{WSF-2} by using the stuffle product as
\begin{align*}
(\Li_1(x)+\Li_1(y))\Li_{k-1}(xy)&=\Li_{1,k-1}(x,xy)+\Li_{k-1,1}(xy,x)+\Li_k(x^2y)\\
&+\Li_{1,k-1}(y,xy)+\Li_{k-1,1}(xy,y)+\Li_k(xy^2),
\end{align*}
and then from Theorem~\ref{Th-1-2} we have
\begin{align*}
& \sum_{j=1}^{k-1}2^{j-1}\left(\Li_{k-j,j}(x^{-1}y,x)+\Li_{k-j,j}(xy^{-1},y)\right)\\
& \quad = -\Li_{1,k-1}(x^{-1},xy)-\Li_{1,k-1}(y^{-1},xy)+\Li_{1,k-1}(x,xy)+\Li_{1,k-1}(y,xy)\\
&\quad\quad  +\Li_{k-1,1}(xy,x)+\Li_{k-1,1}(xy,y)+\Li_k(x^2y)+\Li_k(xy^2)+(k-1)\Li_k(xy)\\
& \quad =\left(\Li_{1,k-1}(y,xy)-\Li_{1,k-1}(x^{-1},xy)\right)+\left(\Li_{1,k-1}(x,xy)-\Li_{1,k-1}(y^{-1},xy)\right) \\
&\quad\quad +\Li_{k-1,1}(xy,x)+\Li_{k-1,1}(xy,y)+\Li_k(x^2y)+\Li_k(xy^2)+(k-1)\Li_k(xy). 
\end{align*}
When $k\ge3$, we may set $y=x^{-1}$ to obtain the corollary. When $k=2$, the identity to be proved is
\[ \Li_{1,1}(x^{-2},x)+\Li_{1,1}(x^2,x^{-1})-\Li_{1,1}(1,x)-\Li_{1,1}(1,x^{-1})=\Li_2(x)+\Li_2(x^{-1})+\zeta(2). \]
This can be directly checked by differentiating with respect to $x$ and noting that both sides are zero
when $x=-1$. 

\section{Various weighted sum formulas}\label{sec-3}

We first deduce equation~\eqref{WSF-OZ} from Corollary~\ref{Cor1-2} by letting $x\to1$. For this, 
we need to show the limit
\[ \lim_{x\to1}\left(\Li_{k-1,1}(x^{-2},x)+\Li_{k-1,1}(x^2,x^{-1})-\Li_{k-1,1}(1,x)-\Li_{k-1,1}(1,x^{-1})\right) =0,\]
for which it is enough to show
\begin{equation} \lim_{x\to1}\left(\Li_{k-1,1}(x^{-2},x)-\Li_{k-1,1}(1,x)\right) =0.\label{limit} \end{equation}
Using the stuffle product, we have
\begin{align*}
&\Li_{k-1,1}(x^{-2},x)-\Li_{k-1,1}(1,x)\\
& \ \ =(\Li_{k-1}(x^{-2})-\Li_{k-1}(1))\Li_1(x)-\Li_{1,k-1}(x,x^{-2})-\Li_{k}(x^{-1}) +\Li_{1,k-1}(x,1)+\Li_{k}(x) \\
& \ \ =(\Li_{k-1}(x^{-2})-\Li_{k-1}(1))\Li_{1}(x)+(\Li_{1,k-1}(x,1)-\Li_{1,k-1}(x,x^{-2}))+(\Li_{k}(x)-\Li_{k}(x^{-1})).
\end{align*}
Suppose $k\ge4$. Then, since 
\begin{align*}
\left|\Li_{k-1}(x^{-2})-\Li_{k-1}(1)\right|&=\left|\sum_{m=1}^\infty \frac{(x^{-2m}-1)}{m^{k-1}}\right|\\
&=\left|x^{-2}-1\right|\left|\sum_{m=1}^\infty \frac{x^{-2(m-1)}+x^{-2(m-2)}+\cdots +x^{-2}+1}{m^{k-1}}\right|\\
&\le \left|x^{-2}-1\right|\sum_{m=1}^\infty \frac{\left|x^{-2(m-1)}+x^{-2(m-2)}+\cdots +x^{-2}+1\right|}{m^{k-1}}\\
&\le \left|x^{-2}-1\right| \zeta(k-2)\\
&=O(x-1)\quad (x\to1), 
\end{align*}
we have 
\[ \lim_{x\to1}(\Li_{k-1}(x^{-2})-\Li_{k-1}(1))\Li_{1}(x) =0,\]
and thus
\[ \lim_{x\to1}\left(\Li_{k-1,1}(x^{-2},x)-\Li_{k-1,1}(1,x)\right) =0.\]
If $k=3$, the well-known reflection relation for $\Li_2(z)$ (see equation~\eqref{reflection} in Section~\ref{sec-4}) gives
\begin{align*}
\Li_2(x^{-2})-\Li_2(1)&=-\Li_2(1-x^{-2})-\log(x^{-2})\,\log(1-x^{-2})\\
& =-(1-x^{-2})\sum_{m=1}^{\infty}\frac{(1-x^{-2})^{m-1}}{m^2}+2(\log x)\log(1-x^{-2}).
\end{align*}
Since $(\log x)\,\log^n(1-x) \to 0$ $(x\to 1)$ for any $n\geq 1$, we have
$$(\Li_{2}(x^{-2})-\Li_{2}(1))\Li_{1}(x)\to 0\quad (x\to 1).$$
Thus we obtain \eqref{limit} and complete the deduction of \eqref{WSF-OZ}.

To obtain the sum formula \eqref{WSF-KT} for $T$-values, we add \eqref{WSF-1} and
\eqref{WSF-1}$|_{x\to y}$, and then subtract \eqref{WSF-2} with $x$ being replaced by $x^{-1}$ 
and also its ${x\leftrightarrow y}$ version. After some rearrangement of terms, we have
\begin{align}
& \sum_{j=2}^{k-1}2^{j-1}\big\{\Li_{k-j,j}(x^2,x^{-1})+\Li_{k-j,j}(x^{-2},x)+\Li_{k-j,j}(y^2,y^{-1})+\Li_{k-j,j}(y^{-2},y) \notag\\
& \quad \ -\Li_{k-j,j}(xy,x^{-1})-\Li_{k-j,j}(x^{-1}y^{-1},x)-\Li_{k-j,j}(xy,y^{-1})-\Li_{k-j,j}(x^{-1}y^{-1},y)\big\} \notag\\
& \quad \ +\left(\Li_{1,k-1}(x^{-1},x^{-1}y)-\Li_{1,k-1}(x,x^{-1}y)\right)+\left(\Li_{1,k-1}(y,x^{-1}y)-\Li_{1,k-1}(y^{-1},x^{-1}y)\right) \notag\\
& \quad \ +\left(\Li_{1,k-1}(x,xy^{-1})-\Li_{1,k-1}(x^{-1},xy^{-1})\right)+\left(\Li_{1,k-1}(y^{-1},xy^{-1})-\Li_{1,k-1}(y,xy^{-1})\right) \notag\\
& \quad \ +\left(\Li_{k-1,1}(x^{-2},x)-\Li_{k-1,1}(1,x)\right)+\left(\Li_{k-1,1}(x^2,x^{-1})-\Li_{k-1,1}(1,x^{-1})\right)\notag\\
& \quad \ +\left(\Li_{k-1,1}(y^{-2},y)-\Li_{k-1,1}(1,y)\right)+\left(\Li_{k-1,1}(y^2,y^{-1})-\Li_{k-1,1}(1,y^{-1})\right) \notag\\
& \quad \ +\left(\Li_{k-1,1}(x^{-1}y,x^{-1})-\Li_{k-1,1}(xy,x^{-1})\right)+\left(\Li_{k-1,1}(x^{-1}y,y)-\Li_{k-1,1}(x^{-1}y^{-1},y)\right)\notag\\
& \quad \ +\left(\Li_{k-1,1}(xy^{-1},x)-\Li_{k-1,1}(x^{-1}y^{-1},x)\right)+\left(\Li_{k-1,1}(xy^{-1},y^{-1})-\Li_{k-1,1}(xy,y^{-1})\right)\notag\\
& \ \  =\Li_k(x)+\Li_k(x^{-1})-\Li_k(x^{-1}y^2)-\Li_k(xy^{-2})\notag\\
& \quad \ +\Li_k(y)+\Li_k(y^{-1}) -\Li_k(x^{-2}y)-\Li_k(x^2y^{-1})\notag \\
& \quad \ +(k-1)\left\{2\zeta(k)-\Li_{k}(x^{-1}y)-\Li_k(xy^{-1})\right\}. \label{WSF-3}
\end{align}
Now we let $(x,y)\to (1,-1)$. Noting that
\begin{align*}
T(k-j,j)&=\sum_{m,n=1}^\infty \frac{(1-(-1)^m)(1-(-1)^n)}{m^{k-j}(m+n)^j}\\
&=\Li_{k-j,j}(1,1)+\Li_{k-j,j}(1,-1)-\Li_{k-j,j}(-1,1)-\Li_{k-j,j}(-1,-1)\quad(j\ge2),\\
T(k)&=\zeta(k)-\Li_k(-1)\quad (k\ge2),
\end{align*}
and the limit  
\[ \lim_{x\to1}\left(\Li_{k-1,1}(x^{-2},x)-\Li_{k-1,1}(1,x)\right) =0\]
as shown before as well as 
\[ \lim_{x\to1}\left(\Li_{k-1,1}(xy^{-1},x)-\Li_{k-1,1}(x^{-1}y^{-1},x)\right)=0\]
which can be similarly proved (we omit it), we obtain \eqref{WSF-KT}.  

Now we proceed to deduce certain sum formulas of level 3 and 4.

Let $\chi_3$ and $\chi_4$ are non-trivial Dirichlet characters of conductor $3$ and $4$ respectively. 
Then the following formulas hold. The stuffle-type double $L$-value $L_*(k_1,k_2;\chi_3,\chi_3)$ is 
defined as
\[ L_*(k_1,k_2;\chi_3,\chi_3)=\sum_{0<m<n}\frac{\chi_3(m)\chi_3(n)}{m^{k_1}n^{k_2}}. \]

\begin{prop}\label{Prop-4}\ For any $k\in \mathbb{Z}_{\geq 2}$,
\begin{align}
& \sum_{j=1}^{k-1}2^{j-1}L_\sh (k-j,j;\chi_3,\chi_3)+L_\sh (k-1,1;\chi_3,\chi_3)\notag\\
&\qquad\qquad\qquad\qquad +L_*(1,k-1;\chi_3,\chi_3)+L_*(k-1,1;\chi_3,\chi_3)\notag \\
& \quad =\frac{k-3}{2}L(k;\chi_3^2)\ \left(=\frac{(k-3)(1-3^{-k})}{2}\zeta(k)\right), \label{WSF-f3}\\
& \sum_{j=1}^{k-1}2^{j-1} L_\sh (k-j,j;\chi_4,\chi_4)+L_\sh (k-1,1;\chi_4,\chi_4)\notag \\
& \quad =\frac{k-1}{2}L(k;\chi_4^2)\ \left(=\frac{(k-1)(1-2^{-k})}{2}\zeta(k)\right). \label{WSF-f4}
\end{align}
\end{prop}

\begin{proof}
First we prove \eqref{WSF-f4}. By $\chi_4(m)=(i^m-(-i)^m)/2i$, we have 
\begin{equation}\label{LbyLi} 4L_\sh(p,q;\chi_4,\chi_4)=\Li_{p,q}(-1,i)+\Li_{p,q}(-1,-i)-\Li_{p,q}(1,i)-\Li_{p,q}(1,-i).\end{equation}
Set $(x,y)=(i,-i)$ in \eqref{WSF-3}. Then we obtain
\begin{align*}
&\sum_{j=2}^{k-1} 2^j\left(\Li_{k-j,j}(-1,i)+\Li_{k-j,j}(-1,-i)-\Li_{k-j,j}(1,i)-\Li_{k-j,j}(1,-i)\right)\\
&\qquad+4\left(\Li_{k-1,1}(-1,i)+\Li_{k-1,1}(-1,-i)-\Li_{k-1,1}(1,i)-\Li_{k-1,1}(1,-i)\right)\\
&=2(k-1)(\zeta(2)-\Li_k(-1)).
\end{align*}
Together with \eqref{LbyLi} and $\zeta(2)-\Li_k(-1)=2L(k;\chi_4^2)$, the identity \eqref{WSF-f4} follows.

As for \eqref{WSF-f3}, we use the relation
\begin{equation}\label{LbyLi3} 3L_\sh(p,q;\chi_3,\chi_3)=\Li_{p,q}(\omega,\omega)+\Li_{p,q}(\omega^{-1},\omega^{-1})
-\Li_{p,q}(1,\omega)-\Li_{p,q}(1,\omega^{-1}) \end{equation}
with $\omega=e^{2\pi i/3}$, which follows from $\chi_3(m)=(\omega^m-\omega^{-m})/\sqrt{3}i$.
Equation \eqref{WSF-f3} follows from \eqref{WSF-3} by setting $(x,y)=(\omega,\omega^{-1})$.
\end{proof}

\begin{remark}
1) The formula \eqref{WSF-f4} was first obtained by M.~Nishi in his master's thesis (in Japanese) submitted to Kyushu University in 2001
(see~\cite[Proposition 4.2]{AK2004}). We note that $L_\sh (k-j,j;\chi_4,\chi_4)$
in Proposition~\ref{Prop-4} is, up to a factor of power of 2, the multiple $\widetilde{T}$-values extensively studied later in \cite{KT2022}.

2) Nishi's  formula for $\chi_3$ is
slightly different from \eqref{WSF-f3} and reads
\begin{align}
& \sum_{j=1}^{k-1}(2^{j-1}+1)L_\sh (k-j,j;\chi_3,\chi_3)+L_\sh (1,k-1;\chi_3,\chi_3)+L_\sh (k-1,1;\chi_3,\chi_3)\notag\\
& \quad =\frac{k-1}{2}L(k;\chi_3^2). \label{Nishi3}
\end{align}
This comes from \eqref{WSF-f3} and a kind of ordinary sum formula
\begin{align}\label{Lsum}
& \sum_{j=1}^{k-1}L_\sh (k-j,j;\chi_3,\chi_3)+L_\sh (1,k-1;\chi_3,\chi_3)-L_*(1,k-1;\chi_3,\chi_3)-L_*(k-1,1;\chi_3,\chi_3)\notag \\
& \quad =L(k;\chi_3^2),
\end{align}
which can be proved by writing $L(1,\chi_3)L(k-1,\chi_3)$ in two ways using shuffle and stuffle products.
Of course we may deduce \eqref{WSF-f3} from Nishi's formula \eqref{Nishi3} and \eqref{Lsum}.
\end{remark}

\section{Proof of Theorem \ref{Th-1-4}}\label{sec-4}
 
We proceed by induction on $k$ to prove the identity (displayed again)
\begin{align} 
\LL_{\scriptsize{\underbrace{1,\ldots,1}_{r-1},k}}\left(\frac{1-x}{1-xy},y\right) &= (-1)^{k-1} \sum_{j_1+\cdots+j_k=r+k\atop
\forall j_i\ge1} \LL_{\scriptsize{\underbrace{1,\ldots,1}_{j_k-1}}}\left(\frac{1-x}{1-xy},y\right)\LL_{j_1,\ldots,j_{k-1}}\left(x,y\right)\notag\\
&\quad +\sum_{j=0}^{k-2} (-1)^j \LL_{\scriptsize{\underbrace{1,\ldots,1}_{r-1},k-j}}(1,y)\LL_{\scriptsize{\underbrace{1,\ldots,1}_{j}}}\left(x,y\right).\label{eq-1-4}
\end{align}

When $k=2$, the identity in question becomes
\begin{equation}\label{k=2}
\LL_{\scriptsize{\underbrace{1,\ldots,1}_{r-1},2}}\left(\frac{1-x}{1-xy},y\right) = -\sum_{j=0}^r \LL_{\scriptsize{\underbrace{1,\ldots,1}_{j}}}\left(\frac{1-x}{1-xy},y\right)\LL_{r+1-j}\left(x,y\right)+\LL_{\scriptsize{\underbrace{1,\ldots,1}_{r-1},2}}(1,y).
\end{equation}
We differentiate both sides with respect to $x$ and check the results are the same. Since both sides are $\LL_{\scriptsize{\underbrace{1,\ldots,1}_{r-1},2}}\left(1,y\right) $ when $x=0$, this confirms the identity.
We use the (easily proved) differential formulas
\[ \frac{\partial}{\partial x}\LL_{k_1,\ldots,k_r}(x,y)=
\begin{cases} 
\dfrac{1}{x}\LL_{k_1,\ldots,k_r-1}(x,y) & k_r>1,\\[2mm]
\dfrac{1-y}{(1-x)(1-xy)}\LL_{k_1,\ldots,k_{r-1}}(x,y) & k_r=1
\end{cases}
\]
and
\[ \frac{\partial}{\partial x}\ \frac{1-x}{1-xy}=-\frac{1-y}{(1-xy)^2}. \]
Using these, we have 
\begin{align*}  \frac{\partial}{\partial x}\LL_{\scriptsize{\underbrace{1,\ldots,1}_{r-1},2}}\left(\frac{1-x}{1-xy},y\right) 
&=\left(\frac{1-x}{1-xy}\right)^{-1}\LL_{\scriptsize{\underbrace{1,\ldots,1}_{r}}}\left(\frac{1-x}{1-xy},y\right) \times
\left(-\frac{1-y}{(1-xy)^2}\right)\\
&=-\frac{1-y}{(1-x)(1-xy)}\LL_{\scriptsize{\underbrace{1,\ldots,1}_{r}}}\left(\frac{1-x}{1-xy},y\right).
\end{align*}
On the other hand, since 
\begin{align*}  \frac{\partial}{\partial x}\LL_{\scriptsize{\underbrace{1,\ldots,1}_{j}}}\left(\frac{1-x}{1-xy},y\right) 
&=\dfrac{1-y}{(1-\frac{1-x}{1-xy})(1-\frac{1-x}{1-xy}y)}\LL_{\scriptsize{\underbrace{1,\ldots,1}_{j-1}}}\left(\frac{1-x}{1-xy},y\right) \times
\left(-\frac{1-y}{(1-xy)^2}\right)\\
&=-\frac1{x}\LL_{\scriptsize{\underbrace{1,\ldots,1}_{j-1}}}\left(\frac{1-x}{1-xy},y\right)
\end{align*}
for $j\ge1$, we have 
\begin{align*}
&\frac{\partial}{\partial x}\left(-\sum_{j=0}^r \LL_{\scriptsize{\underbrace{1,\ldots,1}_{j}}}\left(\frac{1-x}{1-xy},y\right)\LL_{r+1-j}\left(x,y\right)\right)
\\
&=\sum_{j=1}^r \frac1x\LL_{\scriptsize{\underbrace{1,\ldots,1}_{j-1}}}\left(\frac{1-x}{1-xy},y)\right)\LL_{r+1-j}\left(x,y\right)
-\sum_{j=0}^{r-1} \LL_{\scriptsize{\underbrace{1,\ldots,1}_{j}}}\left(\frac{1-x}{1-xy},y\right)\cdot\frac1x \LL_{r-j}\left(x,y\right)\\
&\quad-\LL_{\scriptsize{\underbrace{1,\ldots,1}_{r}}}\left(\frac{1-x}{1-xy},y\right)\cdot \frac{1-y}{(1-x)(1-xy)}\\
&=-\frac{1-y}{(1-x)(1-xy)}\LL_{\scriptsize{\underbrace{1,\ldots,1}_{r}}}\left(\frac{1-x}{1-xy},y\right),
\end{align*}
as expected.
For general $k\ge3$, we have, by using the induction hypothesis, 
\begin{align}&\frac{\partial}{\partial x}\left(\text{L.H.S of }\eqref{eq-1-4}\right)\notag\\
&=-\frac{1-y}{(1-x)(1-xy)}\LL_{\scriptsize{\underbrace{1,\ldots,1}_{r-1},k-1}}\left(\frac{1-x}{1-xy},y\right)\notag\\
&=-\frac{1-y}{(1-x)(1-xy)}\left((-1)^{k-2} \sum_{j_1+\cdots+j_{k-1}=r+k-1\atop\forall j_i\ge1} \right.
\LL_{\scriptsize{\underbrace{1,\ldots,1}_{j_{k-1}-1}}}\left(\frac{1-x}{1-xy},y\right)\LL_{j_1,\ldots,j_{k-2}}\left(x,y\right)\notag\\
&\left. \quad +\sum_{j=0}^{k-3} (-1)^j \LL_{\scriptsize{\underbrace{1,\ldots,1}_{r-1},k-1-j}}(1,y)\LL_{\scriptsize{\underbrace{1,\ldots,1}_{j}}}\left(x,y\right)\right)\label{LHS}
\end{align}
and
\begin{align} 
&\frac{\partial}{\partial x}\left(\text{R.H.S of }\eqref{eq-1-4}\right)\notag\\
&=\frac{\partial}{\partial x}\left((-1)^{k-1}\sum_{j_1+\cdots+j_k=r+k\atop j_k\ge2} 
\LL_{\scriptsize{\underbrace{1,\ldots,1}_{j_k-1}}}\left(\frac{1-x}{1-xy},y\right)\LL_{j_1,\ldots,j_{k-1}}\left(x,y\right)\right.\notag\\
&\left.\quad+(-1)^{k-1} \sum_{j_1+\cdots+j_{k-1}=r+k-1\atop\forall j_i\ge1} 
\LL_{j_1,\ldots,j_{k-1}}\left(x,y\right)+\sum_{j=0}^{k-2} (-1)^j \LL_{\scriptsize{\underbrace{1,\ldots,1}_{r-1},k-j}}(1,y)\LL_{\scriptsize{\underbrace{1,\ldots,1}_{j}}}\left(x,y\right)\right)\notag\\
&=(-1)^{k} \sum_{j_1+\cdots+j_k=r+k\atop j_k\ge2} 
\frac1x \LL_{\scriptsize{\underbrace{1,\ldots,1}_{j_k-2}}}\left(\frac{1-x}{1-xy},y\right)\LL_{j_1,\ldots,j_{k-1}}\left(x,y\right)\label{cancel1}\\
&\quad +(-1)^{k-1}\sum_{j_1+\cdots+j_k=r+k\atop j_k\ge2, j_{k-1}\ge2} 
\LL_{\scriptsize{\underbrace{1,\ldots,1}_{j_k-1}}}\left(\frac{1-x}{1-xy},y\right)\cdot \frac1x \LL_{j_1,\ldots,j_{k-1}-1}\left(x,y\right)\label{cancel2}\\
&\quad +(-1)^{k-1}\sum_{j_1+\cdots+j_k=r+k\atop j_k\ge2, j_{k-1}=1} 
\LL_{\scriptsize{\underbrace{1,\ldots,1}_{j_k-1}}}\left(\frac{1-x}{1-xy},y\right)\cdot \frac{1-y}{(1-x)(1-xy)} \LL_{j_1,\ldots,j_{k-2}}\left(x,y\right)\label{rem1}\\
&\quad+(-1)^{k-1} \sum_{j_1+\cdots+j_{k-1}=r+k-1\atop j_{k-1}\ge2} \frac1x\LL_{j_1,\ldots,j_{k-1}-1}\left(x,y\right)\label{cancel3}\\
&\quad+(-1)^{k-1} \sum_{j_1+\cdots+j_{k-2}=r+k-2\atop \forall j_i\ge1} \frac{1-y}{(1-x)(1-xy)}\LL_{j_1,\ldots,j_{k-2}}\left(x,y\right)\label{rem2}\\
&\quad+\sum_{j=1}^{k-2} (-1)^j \LL_{\scriptsize{\underbrace{1,\ldots,1}_{r-1},k-j}}(1,y)\cdot
\frac{1-y}{(1-x)(1-xy)}\LL_{\scriptsize{\underbrace{1,\ldots,1}_{j-1}}}\left(x,y\right). \label{rem3}
\end{align}
Noting that (by changing $j_k\to j_k+1$)  
\begin{align*} 
& \sum_{j_1+\cdots+j_k=r+k\atop j_k\ge2} 
\frac1x \LL_{\scriptsize{\underbrace{1,\ldots,1}_{j_k-2}}}\left(\frac{1-x}{1-xy},y\right)\LL_{j_1,\ldots,j_{k-1}}\left(x,y\right)\\
&=\sum_{j_1+\cdots+j_k=r+k-1\atop \forall j_i\ge1} 
\frac1x \LL_{\scriptsize{\underbrace{1,\ldots,1}_{j_k-1}}}\left(\frac{1-x}{1-xy},y\right)\LL_{j_1,\ldots,j_{k-1}}\left(x,y\right)
\end{align*}
and
\begin{align*} 
& \!\!\!\sum_{j_1+\cdots+j_k=r+k\atop j_k\ge2, j_{k-1}\ge2} 
\LL_{\scriptsize{\underbrace{1,\ldots,1}_{j_k-1}}}\left(\frac{1-x}{1-xy},y\right) \frac1x \LL_{j_1,\ldots,j_{k-1}-1}\left(x,y\right)
+\!\!\!  \sum_{j_1+\cdots+j_{k-1}=r+k-1\atop j_{k-1}\ge2} \frac1x\LL_{j_1,\ldots,j_{k-1}-1}\left(x,y\right)\\
&=\sum_{j_1+\cdots+j_k=r+k\atop j_{k-1}\ge2} 
\frac1x\LL_{\scriptsize{\underbrace{1,\ldots,1}_{j_k-1}}}\left(\frac{1-x}{1-xy},y\right)  \LL_{j_1,\ldots,j_{k-1}-1}\left(x,y\right)\\
&=\sum_{j_1+\cdots+j_k=r+k-1\atop \forall j_{i}\ge1} 
\frac1x\LL_{\scriptsize{\underbrace{1,\ldots,1}_{j_k-1}}}\left(\frac{1-x}{1-xy},y\right)  \LL_{j_1,\ldots,j_{k-1}}\left(x,y\right),
\end{align*}
we see that the terms \eqref{cancel1}, \eqref{cancel2}, and \eqref{cancel3} add up to 0.
Likewise, the sum of the terms \eqref{rem1}, \eqref{rem2}, and \eqref{rem3}, without the factor $\frac{1-y}{(1-x)(1-xy)}$,
is equal to 
\begin{align*}
&(-1)^{k-1}\sum_{j_1+\cdots+j_k=r+k\atop j_k\ge2, j_{k-1}=1} 
\LL_{\scriptsize{\underbrace{1,\ldots,1}_{j_k-1}}}\left(\frac{1-x}{1-xy},y\right)\LL_{j_1,\ldots,j_{k-2}}\left(x,y\right)\\
&+(-1)^{k-1} \sum_{j_1+\cdots+j_{k-2}=r+k-2\atop \forall j_i\ge1} \LL_{j_1,\ldots,j_{k-2}}\left(x,y\right)
+\sum_{j=1}^{k-2} (-1)^j \LL_{\scriptsize{\underbrace{1,\ldots,1}_{r-1},k-j}}(1,y)\cdot
\LL_{\scriptsize{\underbrace{1,\ldots,1}_{j-1}}}\left(x,y\right)\\
&=(-1)^{k-1}\sum_{j_1+\cdots+j_{k-2}+j_k=r+k-1\atop j_k\ge2} 
\LL_{\scriptsize{\underbrace{1,\ldots,1}_{j_k-1}}}\left(\frac{1-x}{1-xy},y\right)\LL_{j_1,\ldots,j_{k-2}}\left(x,y\right)\\
&+(-1)^{k-1} \sum_{j_1+\cdots+j_{k-2}=r+k-2\atop \forall j_i\ge1} \LL_{j_1,\ldots,j_{k-2}}\left(x,y\right)
+\sum_{j=1}^{k-2} (-1)^j \LL_{\scriptsize{\underbrace{1,\ldots,1}_{r-1},k-j}}(1,y)\cdot
\LL_{\scriptsize{\underbrace{1,\ldots,1}_{j-1}}}\left(x,y\right)\\
&=(-1)^{k-1}\sum_{j_1+\cdots+j_{k-2}+j_{k-1}=r+k-1\atop \forall j_i\ge1} 
\LL_{\scriptsize{\underbrace{1,\ldots,1}_{j_{k-1}-1}}}\left(\frac{1-x}{1-xy},y\right)\LL_{j_1,\ldots,j_{k-2}}\left(x,y\right)\\
&\quad +\sum_{j=0}^{k-3} (-1)^{j+1} \LL_{\scriptsize{\underbrace{1,\ldots,1}_{r-1},k-j-1}}(1,y)\cdot
\LL_{\scriptsize{\underbrace{1,\ldots,1}_{j}}}\left(x,y\right).
\end{align*}
This, multiplied by $\frac{1-y}{(1-x)(1-xy)}$, is equal to \eqref{LHS} and Theorem~\ref{Th-1-4} is proved.
\qed

\begin{example}\label{Rem-5-2}\ We may deduce the famous five-term relation of the dilogarithm function
from \eqref{eq-1-4} in the case $(r,k)=(1,2)$:
\[ \LL_2\left(\frac{1-x}{1-xy},y\right)=-\LL_1\left(\frac{1-x}{1-xy},y\right)\LL_1(x,y)-\LL_2(x,y)+\LL_2(1,y). \]
Namely, since $\LL_{2}(x,y)=\Li_2(x)-\Li_2(xy)$, this can be written as
\begin{align*}
 &\Li_2\left(\frac{1-x}{1-xy}\right)-\Li_2\left(\frac{(1-x)y}{1-xy}\right)+\Li_2(x)-\Li_2(xy) -\Li_2(1)+\Li_2(y)\\
  &=
 -\left(\Li_1\left(\frac{1-x}{1-xy}\right)-\Li_1\left(\frac{(1-x)y}{1-xy}\right)\right)\left(\Li_1(x)-\Li_1(xy)\right). 
\end{align*}
Using $\frac{(1-x)y}{1-xy}=1-\frac{1-y}{1-xy}$ and the reflection formula 
\begin{equation} \Li_2(1-z)=-\Li_2(z)+\Li_2(1)-\log z\log(1-z), \label{reflection} \end{equation}
we may rewrite this as in the form 
\begin{align*}
& \Li_2(x)+\Li_2(y)+\Li_2(1-xy)+\Li_2\left(\frac{1-x}{1-xy}\right)+\Li_2\left(\frac{1-y}{1-xy}\right) \\
&=3\zeta(2)-\log x\log(1-x)-\log y\log(1-y)-\log\left(\frac{1-x}{1-xy}\right)\log\left(\frac{1-y}{1-xy}\right),
\end{align*}
which is presented for instance in \cite[Section 2]{Zagier}\footnote{Note that the constant $\pi^2/6$ there
should be $\pi^2/2$ and the sign in front of $\log((1-x)/(1-xy))$ on the right should be minus.} 
\end{example}

\begin{example}\label{Exam-5-3}\ The case $(r,k)=(2,2)$ of \eqref{eq-1-4} gives
\begin{align*}
& \Li_{1,2}\left(1,\frac{1-x}{1-xy}\right)-\Li_{1,2}\left(y,\frac{1-x}{1-xy}\right) -\Li_{1,2}\left(y^{-1},\frac{(1-x)y}{1-xy}\right)
+\Li_{1,2}\left(1,\frac{(1-x)y}{1-xy}\right) \\
&=\Li_3(1)-\Li_3(x)-\Li_3\left(y \right)+\Li_3\left(xy\right)+(\log x)\left(\Li_2(x)-\Li_2(xy)\right)+\frac12 (\log x)^2 \log\left(\frac{1-x}{1-xy}\right). 
\end{align*}
This can also be obtained from a known formula for $\Li_{2,1}(x,y)$ (see \cite[(2.48)]{Zhao}) and some functional equations for $\Li_3(z)$, 
but the deduction is fairly complicated, and will be omitted.
\end{example}


{\bf Acknowledgements.}\ 
This work was supported by Japan Society for the Promotion of Science, Grant-in-Aid for Scientific 
Research (A) 21H04430 (M. Kaneko), and (C) 21K03168 (H. Tsumura).

\



\begin{thebibliography}{999}
\bibitem{Alf} 
L. V. Ahlfors, Complex Analysis, An Introduction to the Theory of Analytic Functions of One Complex Variable, Third edition, International Series in Pure and Applied Mathematics, McGraw-Hill Book Co., New York, 1978.



\bibitem{AK2004}  T.~Arakawa and M.~Kaneko, 
On multiple $L$-values, {J. Math. Soc. Japan} {\bf 56} (2004), 967--991.

\bibitem{Chapo2022}  F. Chapoton,  
Multiple $T$-values with one parameter, Tsukuba J. Math.  {\bf 46} (2022), 153--163.

\bibitem{GKZ2006}
H. Gangl, M. Kaneko, D. Zagier, Double zeta values and modular forms, in `Automorphic forms and Zeta functions', 
Proceedings of the conference in memory of Tsuneo Arakawa, World Scientific, (2006), 71--106.


\bibitem{Kamano2023}
{K. Kamano}, Poly-Bernoulli numbers with one parameter and their generating functions, Comment. Math. Univ. St. Pauli {\bf 71} (2023), 37--50. 

\bibitem{Kaneko-Tsumura2018}
{M.~Kaneko and H.~Tsumura}, Multi-poly-Bernoulli numbers and related zeta functions, Nagoya Math. J. {\bf 232} (2018), 19--54.

\bibitem{KT2020-ASPM}
{M.~Kaneko and H.~Tsumura}, 
Zeta functions connecting multiple zeta values and poly-Bernoulli numbers, Adv. Stud. Pure Math. {\bf 84}, 2020, 181--204.

\bibitem{KT2020-Tsukuba}
{M.~Kaneko and H.~Tsumura}, On multiple zeta values of level two, Tsukuba J. Math. {\bf 44} (2020), 213--234.

\bibitem{KT2022}
{M.~Kaneko and H.~Tsumura}, Multiple $L$-values of level four, poly-Euler numbers, and related zeta functions, Tohoku Math. J. {\bf 76} (2024), 361--389. 

\bibitem{OZ2008} 
Y. Ohno and W. Zudilin, Zeta stars, Commun. Number Theory Phys. {\bf 2} (2008), 325--347.

\bibitem{Pallewatta}
M.~Pallewatta, 
On polycosecant numbers and level two generalization of Arakawa-Kaneko zeta functions, Doctoral thesis, Kyushu University, 2020.

\bibitem{Sasaki2012}  Y.~Sasaki, 
On generalized poly-Bernoulli numbers and related $L$-functions,
{J. Number Theory} {\bf 132} (2012), 156--170.

\bibitem{Zagier} 
D. Zagier, The dilogarithm function, Frontiers in number theory, physics, and geometry. II, 3--65, Springer, Berlin, 2007.

\bibitem{Zhao}
J. Zhao,
Multiple Zeta Functions, Multiple Polylogarithms and Their Special Values, 
Ser. Number Theory and Its Appl. \textbf{12}, World Sci. Publ., 2016.

\end{thebibliography}
\end{document}